\documentclass{amsart}
\usepackage{extarrows}
\usepackage{tikz-cd}
\usetikzlibrary{matrix}

\newtheorem{theorem}{Theorem}[section]
\newtheorem*{theorem*}{Theorem}
\newtheorem{lemma}[theorem]{Lemma}
 \newtheorem{corollary}[theorem]{Corollary}
 \newtheorem{proposition}[theorem]{Proposition}
 \newtheorem*{mainthm*}{Main Theorem}

\theoremstyle{definition}

\newtheorem{example}[theorem]{Example}

\theoremstyle{remark}
\newtheorem{remark}[theorem]{Remark}

\begin{document}

\title[Mean curvature of hypersurfaces]{Mean Curvature of Hypersurfaces in Killing Submersions with bounded shadow}



\author{Vicent Gimeno}
\address{Department of Mathematics, Universitat Jaume I-IMAC,   E-12071, 
Castell\'{o}, Spain}
\email{gimenov@mat.uji.es}
\thanks{Research partially supported by  the Universitat Jaume I Research Program Project P1-1B2012-18, and DGI-MINECO grant (FEDER) MTM2013-48371-C2-2-PDGI from Ministerio de Ciencia e Inovaci\'{o}n (MCINN), Spain.}




\dedicatory{}

\begin{abstract}
Given a complete hypersurface isometrically immersed in an ambient manifold, in this paper we provide a lower bound for the norm of the mean curvature vector field of the immersion assuming that:

1) The ambient manifold admits a Killing submersion with unit-length Killing vector field.

2)The projection of the image of the immersion is bounded in the base manifold.  

3)The hypersurface is stochastically complete, or the immersion is proper.
\end{abstract}

\maketitle
\section{Introduction}\label{sec:intro}
Lower bounds for the norm of the mean curvature of an isometric immersion of bounded image in the Euclidean space has been largely studied in order to understand the Calabi problem. It is well known that there is no  gap on the norm of the mean curvature of a bounded immersion. 
 More precisely, it is known that a complete isometric immersion in the Euclidean space with bounded image can be a minimal immersion. For instance,  in his celebrated paper \cite{Nadi}, Nadirashvili  constructed a complete minimal surface
inside a round ball in $\mathbb{R}^3$. Later on the construction of minimal immersions inside of  bounded domains of the Euclidean space was carried out by  Martin,
Morales, Tokuomaru, Alarc\'on, Ferrer and Meeks among others  (see \cite{Morales1,Morales2,Tokuomaru2007,Alarc2008,Ferrer2012}). More recently, Alarc\'on and Forstneri\v c have proved in \cite{Alarc2015} that every bordered Riemann surface carries a conformal complete minimal immersion into $\mathbb{R}^3$ with bounded image. 

Despite of this freedom in the norm of the mean curvature we want stress here that all this previous examples are geodesically complete but stochastically incomplete and are non-properly immersed in the Euclidean space. A Riemmanian manifold it is said stochastically complete if $e^{t\triangle }1=1$ for all $t\geq 0$ where $\triangle={\rm div}\nabla$ is the Laplacian and $\{P_t=e^{t\triangle}\}_{t\geq 0}$ is the heat semi-group of the Laplacian (see section \ref{stochastic} for a more detailed description of stochastic completeness).

Historically, the first attempt to construct a minimal immersion with bounded image dealt with the construction of an example with bounded projection. In 1980, (before the example of Nadirashvili),   Jorge and Xavier in \cite{JorgeXavier80} exhibit a non-flat complete minimal surface lying between two parallel planes. This example is stochastically complete (see \cite{Atsuji1999}) but is non-properly immersed.

The relevant point here is that stochastic completeness or the properness of the isometric immersion implies lower bounds for the norm of the mean curvature of bounded immersions.
In fact, in \cite{Alias2009}, Al\'ias, Bessa and Dajczer  proved that given a complete isometric immersion $\alpha:\Sigma^{n-1}\to\mathbb{R}^{n}$ with $\Sigma$ stochastically complete and with bounded projection of the image of $\alpha$ in a $2$-plane, namely, there exists a geodesic ball $B_R^{\mathbb{R}^2}$ of radius $R$  and a projection $\pi:\mathbb{R}^{n}\to \mathbb{R}^2$ such that $\pi(\varphi(\Sigma))\subset B_R^{\mathbb{R}^2}$,   the norm of the mean curvature $\vec H$ of the immersion is bounded by
$$
\sup_M\Vert \vec{H}\Vert\geq h(R)
$$
where $h(R)$ is the norm of the mean curvature of the generalized cylinder $\pi^{-1}(\partial B_R^{\mathbb{R}^2})$ in $\mathbb{R}^n$. More generally, in \cite{Alias2009}, it is proved that
\begin{theorem}[See \cite{Alias2009}]Let $\varphi: \Sigma^m\to N^{n-l}\times \mathbb{R}^l$ be an isometric immersion of a complete Riemannian manifold $\Sigma$ of dimension $m>l+1$. Let $B^N_R(p)$ be the geodesic  ball of $N^{n-l}$ centered at $p$ with radius $R$. Given $q\in \Sigma$, assume that the radial sectional curvature $K^{\rm rad}_N$ along the radial geodesics issuing from $p=\pi_N(\varphi(q))\in N^{n-1}$ is bounded as $K^{\rm rad}_N\leq \kappa$ in $B_R^N(p)$. Suppose that
$$
\varphi(\Sigma)\subset B^N_R(p)\times \mathbb{R}^l
$$
for $R<\min\{{\rm inj}_N(p),\pi/2\sqrt{\kappa}\}$, where we replace $\pi/2\sqrt{\kappa}$ by $+\infty$ if $\kappa\leq 0$. Then, if $\Sigma$ is stochastically complete or $\varphi:\Sigma\to N^{n-l}\times \mathbb{R}^l$, the supremum of the norm of the mean curvature vector field is bounded by 
$$
\sup_\Sigma\Vert\vec H\Vert\geq \frac{m-l}{m}C_{\kappa}(R)
$$
where 
$$
C_\kappa(R)=\left\{\begin{array}{lcl}
\sqrt{\kappa}\cot(\sqrt{\kappa}R)&{\rm if}& \kappa>0,\, R<\frac{\pi}{2\sqrt{\kappa}}\\
1/R&{\rm if }&\kappa=0\\
\sqrt{-\kappa}\coth(\sqrt{-\kappa}R)&{\rm if}& \kappa<0.
\end{array}\right.
$$
\end{theorem}

In this paper we are interested in a similar case but when we have an isometric immersion $\varphi:\Sigma^n\to M^{n+1}$ and the ambient manifold $M$ admits a Killing submersion $\pi:M^{n+1}\to \mathbb{B}^n$ (see sections \ref{mainsec} and \ref{prelimsec}). Our goal is to obtain lower bounds for the mean curvature of the immersion $\varphi$ when the projection of the image $\varphi(\Sigma)$ is bounded, \emph{i.e.}, there exist a geodesic ball $B_R^{\mathbb{B}}$ in $\mathbb{B}$ such that $\varphi(\Sigma)\subset \pi^{-1}(B_R^\mathbb{B})$.  In our main results, we prove that if we assume that $\Sigma$ is stochastically complete, or the immersion is proper, then the norm of the mean curvature is bounded from below by a function that depends on $R$, on an upper bound for the sectional curvatures of $B_R^\mathbb{B}$, and on the bundle curvature of the Killing submersion $\pi:M^{n+1}\to \mathbb{B}$.

\subsection*{Outline of the paper}In \S \ref{mainsec} are stated the main results of this paper; theorem \ref{maintheo} and theorem \ref{teo2}. In \S \ref{prelimsec} are developed every necessary lemma, proposition, and theorems in order to prove in \S \ref{sec4}  theorem \ref{maintheo} and in \S \ref{sec5} theorem \ref{teo2}. 

\section{Main Results}\label{mainsec}
Let $(M,g_M)$ be $(n+1)$-dimensional Riemannian manifold. It is said that $M$ admits a Killing submersion with an unit-length Killing vector field if there exist a Riemmanian submersion $\pi:(M,g_M)\to (\mathbb{B},g_\mathbb{B})$ into a $n$-dimensional base manifold $\mathbb{B}$, such that the fibers of the submersion are integral curves of an unit-length Killing vector field $\xi\in \mathfrak{X}(M)$.  The tangent space $T_pM$ at any point $p\in M$ can be decomposed $T_pM=\mathcal{V}(p)\oplus\mathcal{H}(p)$ in its vertical $\mathcal{V}(p)={\rm ker}(d\pi_p)$ and horizontal part $\mathcal{H}(p)=({\rm ker}(d\pi_p))^\perp$ respectively.

The Killing vector field $\xi$ induces a smooth $(1,1)$-tensor field $\nabla \xi$, see section \ref{prelimsec}, given by
$$
\nabla \xi(p):\mathcal{H}(p)\to\mathcal{H}(p),\quad v\to\nabla_v\xi,
$$ 
where $\nabla$ denotes the Levi-Civita connection in $M$.  With the Hilbert-Schmidt norm $\Vert \nabla \xi(p)\Vert^2$ of the linear map $\nabla \xi(p)$ we can define the \emph{$\tau$-function} as
$$
\tau(p) :=\left(\frac{\Vert \nabla \xi(p)\Vert^2}{2}\right)^{\frac{1}{2}}
$$
For any point $p\in M$ and any $v\in \mathcal{H}(p)$, the sectional curvature of the plane $v\wedge \xi$ spanned by $v$ and $\xi$ is non-negative and bounded from above by (see section \ref{prelimsec})
$$
{\rm sec}(v\wedge \xi)\leq \tau^2(p).
$$ 
The above inequality is an equality if ${\rm dim}(\mathbb{B})=2$ and in such a case $\tau$ is also known\footnote{In the case of ${\rm dim}(\mathbb{B})=2$ usually is used a signed $\tau$. In such convention our $\tau$ is just $\vert \tau\vert$.} as the \emph{bundle curvature of the submersion} and $\tau^2$ coincides with the sectional curvature of the vertical planes.

In this paper we are interested in hypersurfaces $\varphi:\Sigma\to M$ isometrically immersed in $M$, and such that $\pi\circ\varphi(\Sigma)$ is bounded in $\mathbb{B}$. Namely, there exists a geodesic ball $B_R^{\mathbb{B}}(o)$ in $\mathbb{B}$ such that $\varphi(\Sigma)\subset \pi^{-1}(B_R^{\mathbb{B}}(o))$.  In this case the restriction $\tau\circ \varphi$  to the hypersurface of the $\tau$-function is absolutely bounded and we denote by
$$
\tau_\Sigma:=\sup_{\Sigma}  \tau\circ\varphi.
$$
In the statement of theorem \ref{maintheo}, we  compare the norm of the mean curvature of $\Sigma$ with the norm of the mean curvature of the inclusion map of the cylinder $\partial B_R^{\mathbb{M}^n(\kappa)}\times \mathbb{R}$  in $\mathbb{M}^n(\kappa)\times \mathbb{R}$, where $\mathbb{M}^n(\kappa)$ is the $n$-dimensional simply -connected space form of sectional curvature $\kappa$, \emph{i.e.},
$$
\mathbb{M}^n(\kappa):=\left\{\begin{array}{lcl}
\mathbb{S}^{n}(\kappa)&{\rm if}&\kappa>0\\
\mathbb{R}^{n}&{\rm if}&\kappa=0\\
\mathbb{H}^{n}(\kappa)&{\rm if}& \kappa<0.
\end{array}
\right.
$$
Our main theorem is the following
\begin{theorem}\label{maintheo}
Let $\Sigma$ be a complete and non-compact Riemannian manifold. Let $\varphi:\Sigma\to M$ be an isometric immersion. Suppose that $M$ admits a Killing submersion $\pi:M\to \mathbb{B}$ with unit-length Killing vector field, suppose moreover that $\varphi(\Sigma)\subset \pi^{-1}(B_R^{\mathbb{B}}(o))$ for some geodesic ball $B_R^{\mathbb{B}}(o)$ of radius $R$ centered at $o\in \mathbb{B}$. Assume that the sectional curvatures are bounded ${\rm sec }\leq \kappa$ in $B_R^{\mathbb{B}}(o)$ and that
$$
R<\min\left\{{\rm inj}(o),\,\frac{\pi}{2\sqrt{\kappa}}\right\}
$$
where ${\rm inj}(o)$ is the injectivity radius of $o$ and we replace $\pi/2\sqrt{\kappa}$ by $+\infty$ if $\kappa<0$. Then, if $\Sigma$ is stochastically complete, the supremum of the norm of the mean curvature vector field of $\Sigma$ satisfies 
$$
\sup_\Sigma\Vert \vec H\Vert\geq h_{\kappa}^n(R)-\frac{\tau_\Sigma}{n}
$$
where $h_{\kappa}^n(R)$ is the norm of the mean curvature of the generalized cylinder $\partial B^{\mathbb{M}^n(\kappa)}_R\times \mathbb{R}$ in $\mathbb{M}^n(\kappa)\times \mathbb{R}$.
\end{theorem}
Because $\lim_{t\to 0}h^n_\kappa(t)=+\infty$, as an immediate corollary of the main theorem we can state
\begin{corollary}Let $\pi:M\to \mathbb{B}$ be a Killing submersion with a Killing vector field of unit-length, suppose that $\mathbb{B}$ has bounded geometry, \emph{i.e.},
\begin{enumerate}
\item The injectivity radius ${\rm inj}(\mathbb{B})$ of $\mathbb{B}$ is positive, $r_{\rm inj}:={\rm inj}(\mathbb{B})>0$.
\item The sectional curvatures of $\mathbb{B}$ are bounded from above by a positive constant, ${\rm sec}(\mathbb{B})\leq \kappa<0$.
\end{enumerate}
Suppose moreover that,
$$
\tau_M:=\sup_M\tau<\infty
$$
Then there exists a constant $R_c(r_{\rm inj},\kappa,\tau_M)\in \mathbb{R}_+\cup \{+\infty\}$ depending only on $r_{\rm inj}$, $\kappa$, and $\tau_M$ such that if a complete Riemannian manifold admits an isometric minimal immersion $\varphi:\Sigma\to \pi^{-1}(B^{\mathbb{B}}_{R_c(r_{\rm inj},\kappa,\tau_M)}(x))$ for some $x\in \mathbb{B}$, then $\Sigma$ is stochastically incomplete.
\end{corollary}

\begin{remark}
From the main theorem we can give an estimate for $R_c(r_{\rm inj},\kappa,\tau_M)$ as 
$$
\min\left\{r_{\rm inj}\,,\,\sup_{0<t<\frac{\pi}{2\sqrt{\kappa}}}\left\{t\, :\, h_{\kappa}^n(t)>\frac{\tau_M}{n}\right\}\right\}.
$$ 
We must remark here that we do not claim that this estimate is sharp.
\end{remark}
\begin{example}[Application of the theorem \ref{maintheo} to the $\mathbb{E}(\kappa,\tau)$ spaces]
A simply-connected homogeneous $3$-dimensional space with $4$-dimensional isometry group is always a Riemmanian fibration over to  a simply-connected $2$-dimensional real space form $\mathbb{M}^2(\kappa)$ and the fibers are integral curves of a unit Killing field, see \cite{Daniel2007,Scott1983}. Namely, every  simply-connected homogeneous $3$-dimensional space with $4$-dimensional isometry group admits a Killing submersion with unit-length Killing vector field. In fact, this spaces can be classified, up to isometries, by their values $\kappa$ and $\tau$ and can be denoted as $\mathbb{E}(\kappa,\tau)$ spaces. 

When the bundle curvature vanishes $\tau=0$ and $\kappa\neq 0$ we obtain the Riemannian products $\mathbb{S}^2(\kappa)\times \mathbb{R}$ for $\kappa>0$, and $\mathbb{H}^2(\kappa)\times \mathbb{R}$ for $\kappa<0$. In the case of $\tau\neq 0$ we have Berger spheres for $\kappa>0$, the Heisenberg group ${\rm Nil}_3$ for $\kappa=0$ and the universal cover $\widetilde{PSL_2(\mathbb{R})}$ of $PSL_2(\mathbb{R})$ for $\kappa<0$.
 
If we have a complete and stochastically complete surface $\Sigma$ immersed in $\mathbb{E}(\kappa,\tau)$ inside of $\pi^{-1}(B_R^{\mathbb{M}^2(\kappa)})$, with $R<\pi/{2\sqrt{\kappa}}$ when $\kappa>0$, then by the main theorem the norm of the mean curvature vector field satisfies
$$
\sup_{\Sigma}\Vert \vec H\Vert\geq h_\kappa^2(R)-\frac{\tau}{2}=\frac{1}{2}\left(C_\kappa(R)-\tau\right).
$$
In the particular case when $\kappa<0$,
$$
\sup_{\Sigma}\Vert \vec H\Vert\geq h_\kappa^2(R)-\frac{\tau}{2}>\frac{1}{2}\left(\sqrt{-\kappa}-\tau\right).
$$
And therefore if $-\kappa\geq \tau^2\geq 0$, any minimal surface immersed in $\pi^{-1}(B_R^{\mathbb{M}^2_\kappa})$ for any $R>0$ is stochastically incomplete. 
\end{example}
In the case when the Killing submersion admits a smooth section such that the normal exponential map is a diffeomorphism, the lower bound for the norm of the mean curvature vector field can be improved replacing the hypothesis on the stochastic completeness of $\Sigma$ in theorem \ref{maintheo} by the properness of the immersion as it is stated in the following theorem 
\begin{theorem}\label{teo2}Let $\Sigma$ be a complete and non-compact Riemannian manifold. Let $\varphi:\Sigma\to M$ be a proper isometric immersion. Suppose that $M$ admits a Killing submersion $\pi:M\to \mathbb{B}$ with unit-length Killing vector field, suppose moreover that $\varphi(\Sigma)\subset \pi^{-1}(B_R^{\mathbb{B}}(o))$ for some geodesic ball $B_R^{\mathbb{B}}(o)$ of radius $R$ centered at $o\in \mathbb{B}$. Assume that the sectional curvatures are bounded ${\rm sec }\leq \kappa$ in $B_R^{\mathbb{B}}(o)$ and that
$$
R<\min\left\{{\rm inj}(o),\,\frac{\pi}{2\sqrt{\kappa}}\right\}
$$
where ${\rm inj}(o)$ is the injectivity radius of $o$ and we replace $\pi/2\sqrt{\kappa}$ by $+\infty$ if $\kappa<0$.  Suppose moreover, that $\pi$ admits a smooth section $s:\overline{B_R^\mathbb{B}(o)}\to M$ and the normal exponential map
$$
\exp: s(B_R^\mathbb{B}(o))\times \mathbb{R}\to \pi^{-1}((B_R^\mathbb{B}(o))),\quad (p,z)\to \exp(p,z):=\exp_p(z\xi)
$$
is a diffeomorphism.  Then,  the supremum of the norm of the mean curvature vector field of $\Sigma$ satisfies 
$$
\sup_\Sigma\Vert \vec H\Vert\geq h_{\kappa}^n(R)
$$
where $h_{\kappa}^n(R)$ is the norm of the mean curvature of the generalized cylinder $\partial B^{\mathbb{M}^n(\kappa)}_R\times \mathbb{R}$ in $\mathbb{M}^n(\kappa)\times \mathbb{R}$.
\end{theorem}

In the case of $\mathbb{E}(\kappa,\tau)$ spaces with $\kappa\leq 0$ we can use the model for the $\mathbb{E}(\kappa,\tau)$ as  (see \cite{Manzano2017}) the space
$$
\mathbb{E}(\kappa,\tau)=\left\{(x,y,z)\in \mathbb{R}^3\,:\, 1+\frac{\kappa}{4}(x^2+y^2)>0 \right\}
$$
endowed with the Riemannian metric such that the following three vector fields
$$
\begin{aligned}
&E_1=\left[1+\frac{\kappa}{4}(x^2+y^2)\right]\frac{\partial}{\partial x}-\tau y\frac{\partial}{\partial z},\\& E_2=\left[1+\frac{\kappa}{4}(x^2+y^2)\right]\frac{\partial}{\partial y}+\tau x\frac{\partial}{\partial z},\\ &E_3=\frac{\partial}{\partial z}
\end{aligned}
$$
constitutes an orthonormal basis in each tangent space. Observe that 
$\pi(x,y,z)\to (x,y)$ is a Riemannian submersion from $\mathbb{E}(\kappa,\tau)$ to $\mathbb{M}^2(\kappa)$ whose fibers are the integral curves of the unit-length Killing vector field $E_3$. Moreover,  
$
s(x,y)\to (x,y,0)
$
constitutes a smooth global section from $\mathbb{M}^2(\kappa)$ to $\mathbb{E}(\kappa,\tau)$. The normal exponential map satisfies
$$
\exp((x,y,0),t)=\exp_{(x,y,0)}(t E_3)=(x,y,t).
$$
In this spaces the hypothesis of theorem \ref{teo2} are therefore fulfilled and hence we can state
\begin{corollary}Let $\varphi:\Sigma\to \mathbb{E}(\kappa,\tau)$ be a proper isometric immersion from the complete and non-compact surface to  $\mathbb{E}(\kappa,\tau)$ with $\kappa\leq 0$. Suppose that $\pi(\varphi(\Sigma))$ is contained in some ball $B_R^{\mathbb{M}^2(\kappa)}(o)$ of radius $R$ in $\mathbb{M}^2(\kappa)$, then the supremum of the norm of the mean curvature vector field of $\Sigma$ satisfies
$$
\sup_\Sigma\Vert \vec{H}\Vert\geq h^2_\kappa(R) 
$$
where $h_{\kappa}^2(R)$ is the norm of the mean curvature of the generalized cylinder $\partial B^{\mathbb{M}^2(\kappa)}_R\times \mathbb{R}$ in $\mathbb{M}^2(\kappa)\times \mathbb{R}$.
\end{corollary}
\begin{remark} The $\mathbb{E}(\kappa,\tau)$ spaces includes for $\tau=0$, $\mathbb{E}(\kappa<0,\tau=0)=\mathbb{H}^2(\kappa)\times \mathbb{R}$ and $\mathbb{E}(\kappa=0,\tau=0)=\mathbb{R}^2\times\mathbb{R}=\mathbb{R}^3$. In this cases the above corollary is a direct application of \cite{Alias2009}. For the case of $\tau\neq 0$ we have the Heisengerb group ${\rm Nil}_3=\mathbb{E}(\kappa=0,\tau\neq 0)$ for $\kappa=0$ and in the case of negative curvature $\kappa<0$ we can assume up to scaling that we are in 
universal cover of $PSL_2(\mathbb{R})$ , namely $\widetilde{PSL}_2(\mathbb{R})=\mathbb{E}(\kappa=-1,\tau\neq 0)$.  By using the above corollary,  any properly immersed non-compact surface $\varphi:\Sigma\to{\rm Nil_3}$ with bounded projection $\pi(\varphi(\Sigma))\subset B_R^{\mathbb{R}^2}(o)\subset\mathbb{R}^2$, has bounded from below the supremum of the norm of the mean curvature vector field by
\begin{equation}\label{novaeq1}
\sup_\Sigma\Vert \vec{H}\Vert\geq \frac{1}{2R}.
\end{equation}
In the case of negative curvature if we have a complete and non-compact surface $\Sigma$ properly immersed in $\widetilde{PSL}_2(\mathbb{R})$ with bounded projection $\pi(\varphi(\Sigma))\subset B_R^{\mathbb{H}^2(-1)}\subset \mathbb{H}^2(-1)$, then the surface has bounded from below the norm of the mean curvature vector field by
\begin{equation}\label{novaeq2}
\sup_\Sigma\Vert \vec{H}\Vert\geq \frac{1}{2}{\rm cotanh}(R).
\end{equation}
Observe that inequalities (\ref{novaeq1}) and (\ref{novaeq2}) are optimal because the right side coincides with the norm of the mean curvature of the cylinders $\pi^{-1}(\partial B_R^{\mathbb{R}^2})\subset {\rm Nil}_3$ and $\pi^{-1}(\partial B_R^{\mathbb{H}^2(-1)})\subset \widetilde{PSL}_2(\mathbb{R})$  respectively.
\end{remark}
\section{Preliminaries}\label{prelimsec}

\subsection{Killing Submersions}
Let $M$ and $\mathbb{B}$ two manifolds. A \emph{ submersion} $\pi:M\to \mathbb{B}$ is a mapping of $M$ onto $\mathbb{B}$ such that its derivative $d\pi_p:T_pM\to T_{\pi(p)}\mathbb{B}$ has maximal rank (it is onto) for any $p\in M$. Then, the distribution $p\to \mathcal{V}(p)={\rm ker}(d\pi_p)$ called the vertical distribution is an involutive  distribution and hence  $\pi^{-1}(x)$ is a submanifold of $M$ of dimension ${\rm dim}(M)-{\rm dim}(\mathbb{B})$ for any $x\in\mathbb{B}$. The submanifolds $\pi^{-1}(x)$ are  called the \emph{fibers}. A vector field $X\in\mathfrak{X}(M)$ is called vertical if it belongs to $\mathcal{V}$, namely, if $X(p)\in \mathcal{V}(p)$ for any $p\in M$.

If $(M,g)$ is moreover a Riemannian manifold, an other distribution called the horizontal distribution can be constructed as $p\to \mathcal{H}(p)=({\rm ker}(d\pi_p))^\perp$. Likewise, a vector field $X\in \mathfrak{X}(M)$ is called horizontal if it belongs to $\mathcal{H}$. Then, for any $p\in M$ we can decompose the tangent space $T_pM$ as
$$
T_pM=\mathcal{H}(p)\oplus\mathcal{V}(p).
$$
A \emph{Riemannian submersion} $\pi:(M, g_M)\to (\mathbb{B},g_{\mathbb{B}})$ is a submersion  such that $d\pi$ preserves the lengths of horizontal vectors. Namely $d\pi_p$ is a local isometry from $\mathcal{H}(p)$ to $T_{\pi(p)}\mathbb{B}$.

A Riemannian submersion $\pi:M\to\mathbb{B}$ is a \emph{Killing submersion} if the fibers $\pi^{-1}(x)$ for any $x\in \mathbb{B}$ are integral curves of a Killing vector field $\xi\in \mathfrak{X}(M)$. Along this paper is assumed that the Killing vector field $\xi$ is an unit-length vector field ($\Vert \xi\Vert=1$). See \cite{Lerma2017} for the general discussion of a Killing submersion with a Killing vector field of non-constant norm. 

Recall that a vector field $\xi\in \mathfrak{X}(M)$ is a \emph{Killing vector field} of $(M,g)$ (see \cite{Oneill}) if its  Lie derivative of the metric tensor vanishes identically, $\mathcal{L}_\xi(g)=0$. If $\xi$ is a Killing vector field, the metric tensor does not change under the flow of $\xi$ and $\xi$ generates local isometries.

The following proposition of a Killing vector field will be used along this paper in order to characterize a Killing vector field

\begin{proposition}[See \cite{Oneill}]\label{killingchar}
Let $(M,g)$ be a Riemannian manifold. Then, the following conditions are equivalents for a vector field $\xi\in\mathfrak{X}(M)$
\begin{enumerate}
\item $\xi$ is Killing; that is, $\mathcal{L}_\xi g=0$.
\item $\nabla \xi$ is skew-adjoint relative to $g$; that is, $\langle \nabla_V\xi,W\rangle= -\langle \nabla_W\xi,X\rangle$ for all $V,W\in\mathfrak{X}(M)$.
\end{enumerate}
\end{proposition}
 
If $\pi:M\to\mathbb{B}$ is a Killing submersion,
for any $p\in M$, by using the vertical vector field $\xi$,  the following linear map $\nabla \xi:T_pM\to T_pM,\quad v\to\nabla\xi(v)=\nabla_v\xi$ can be defined. Since $\xi$ is a unit-length Killing vector field, by proposition \ref{killingchar}
\begin{equation}\label{eq1}
\langle \nabla_v\xi,w\rangle=-\langle \nabla_w\xi,v\rangle
\end{equation}
for any $v,w\in T_pM$. This implies that $\pi^{-1}(x)$ is geodesic in $M$ because $\pi^{-1}(x)$ is the integral curve of $\xi$, and by using that $\Vert \xi\Vert=1$ and equality (\ref{eq1}), we conclude
\begin{equation}\label{eq2}
\langle \nabla_\xi\xi,v\rangle=-\langle \nabla_v\xi,\xi\rangle=-\frac{1}{2}v\langle\xi,\xi\rangle=0,
\end{equation}
for any $v\in T_pM$. Therefore, $\nabla_\xi\xi=0$, and as we have stated $\pi^{-1}(x)$ is a geodesic in $M$. Moreover, from equation (\ref{eq2}) we deduce that $\nabla_v\xi$ is perpendicular to $\xi$, and hence horizontal. The restriction of $\nabla\xi$ to $\mathcal{H}(p)$ induces therefore a linear map $\nabla \xi(p):\mathcal{H}(p)\to\mathcal{H}(p)$.

In the following proposition it is summarized the properties of the $(1,1)$-tensor field $\nabla\xi$ and of $\tau=\left(\frac{\Vert \nabla \xi\Vert^2}{2}\right)^{\frac{1}{2}}$ that are relevant for the present paper
\begin{proposition}\label{propunua}
Let $\pi:M\to \mathbb{B}$ be a Killing submersion with unit-length Killing vector field. Then
\begin{enumerate}
\item Given a point $p\in M$ and an horizontal vector $v\in T_pM$, the sectional curvature ${\rm sec}(v\wedge \xi(p))$ of the plane spaned by $\xi(p)$ and $v$ is bounded by 
$$
{\rm sec}(v\wedge \xi(p))\leq \tau^2(p)
$$
with equality if ${\rm dim}(\mathbb{B})=2$.
\item Given a point $p\in M$ and for any horizontal vector $v\in T_pM$ $$\Vert \nabla_v\xi\Vert^2\leq \tau^2\Vert v\Vert^2$$
\item The function $\tau:M\to \mathbb{R}$ is a basic function, \emph{i.e.}, it is fiber-independent, namely, if $\pi(x)=\pi(y)$ then $\tau(x)=\tau(y)$. 
\end{enumerate}
\end{proposition}
\begin{proof}
Given a point $p\in M$ and an horizontal vector $v\in \mathcal{H}(p)$ with unit-length, $\Vert v\Vert=1$, in order to obtain the sectional curvature ${\rm sec}(v\wedge \xi)$ let us consider a vector field $\overline X\in\mathfrak{X}(\mathbb{B})$ defined in a neighborhood $U\ni\pi(p)$, such that $\overline X(\pi(p))=d\pi(v)$ and with vanishing covariant derivative $\nabla^\mathbb{B}_{\overline X}\overline X=0$  in $\mathbb{B}$, \emph{i.e.}, a geodesic vector field.   Then the lift $X\in \mathfrak{X}(M)$ of $\overline X$ defined in $\pi^{-1}(U) \ni p$ satisfies
\begin{equation}
\left\{
\begin{array}{rcl}
d\pi(X)&=&d\pi(X^H)=\overline X\\
d\pi((\nabla_XX)^H)&=&\nabla^B_{\overline X}\overline X=0\\
\nabla_XX&=&(\nabla_XX)^H+\langle \nabla_XX,\xi\rangle\xi=-\langle X,\nabla_X\xi\rangle=0.
\end{array}\right.
\end{equation}
Where here and in what follows the superscript $H$ denotes the horizontal part of a vector.  Then,
$$
\begin{aligned}
{\rm sec}(v\wedge \xi)=&{\rm sec}(X\wedge \xi)=\langle R(X,\xi)X,\xi\rangle=\langle \nabla_\xi\nabla_XX-\nabla_X\nabla_\xi X+\nabla_{[X,\xi]}X,\xi\rangle\\
=&\langle-\nabla_X\nabla_\xi X+\nabla_{[X,\xi]}X,\xi\rangle=\langle-\nabla_X\left([\xi, X]+\nabla_X\xi\right)+\nabla_{[X,\xi]}X,\xi\rangle\\
=&\langle\nabla_X\left([X,\xi]-\nabla_X\xi\right)+\nabla_{[X,\xi]}X,\xi\rangle
\end{aligned}
$$
In order to simplify the expression let us define the following vector fields $Y:=\nabla_X\xi$ and $Z:=[X,\xi]$.  Observe that both $X,Y$ are horizontal vector fields. Since $Y=\nabla\xi(X)$ and
$$
\langle Z,\xi\rangle=\langle \nabla_X\xi-\nabla_\xi X,\xi\rangle=\langle \nabla_X\xi,\xi\rangle-\langle\nabla_\xi X,\xi\rangle=\frac{1}{2}X\langle\xi,\xi\rangle=0
$$
Therefore,
$$
\begin{aligned}
{\rm sec}(v\wedge \xi)=&\langle\nabla_X\left(Z-Y\right)+\nabla_{Z}X,\xi\rangle=\langle\nabla_XZ+\nabla_Z X,\xi\rangle-\langle\nabla_XY,\xi\rangle\\
=&\Vert Y\Vert^2=\Vert \nabla_X\xi\Vert^2=\Vert \nabla_v\xi\Vert^2
\end{aligned}
$$
where we have used that
$
\langle \nabla_XZ,\xi\rangle=-\langle Z,\nabla_X\xi\rangle=\langle X,\nabla_Z\xi\rangle=-\langle \nabla_ZX,\xi\rangle
$
and
$
\langle\nabla_XY,\xi\rangle=-\langle Y,\nabla_X\xi\rangle=-\Vert Y\Vert^2.
$

In order to obtain item (1) and (2) of the proposition we only need to relate $\Vert \nabla \xi(v)\Vert^2$ with $\Vert \xi\Vert^2$. 
When we focus on $p\in M$ and consider an orthonormal basis $\{E_i\}_{i=1}^n$ of $\mathcal{H}(p)$, for any $v\in \mathcal{H}(p)$, $v=\sum_iv^iE_i$, 
\begin{equation}\label{eq2.4}
\begin{aligned}
\Vert \nabla \xi(v)\Vert^2=&\Vert\nabla_v\xi\Vert^2=\sum_{i=1}^n\langle \nabla_v\xi,E_i\rangle^2=\sum_{i=1}^n\sum_{j=1}^n(v^j)^2\langle \nabla_{E_j}\xi,E_i\rangle^2\\
=&\sum_{i=1}^n\sum_{j=1}^n(v^j)^2\left(\frac{\langle \nabla_{E_j}\xi,E_i\rangle^2+\langle \nabla_{E_i}\xi,E_j\rangle^2}{2}\right)\\
=&\sum_{i=1}^n\sum_{j=1}^n(v^j)^2\frac{\langle \nabla_{E_j}\xi,E_i\rangle^2}{2}+\sum_{i=1}^n\sum_{j=1}^n(v^j)^2\frac{\langle \nabla_{E_i}\xi,E_j\rangle^2}{2}\\
=&\sum_{i=1}^n\sum_{j=1}^n(v^j)^2\frac{\langle \nabla_{E_j}\xi,E_i\rangle^2}{2}+\sum_{j=1}^n\sum_{i=1}^n(v^i)^2\frac{\langle \nabla_{E_j}\xi,E_i\rangle^2}{2}\\
=&\sum_{i=1}^n\sum_{j=1}^n\frac{(v^j)^2+(v^i)^2}{2}\langle \nabla_{E_j}\xi,E_i\rangle^2\\
=&\sum_{i=1}^{n-1}\sum_{j>1}^n\left((v^j)^2+(v^i)^2\right)\langle \nabla_{E_j}\xi,E_i\rangle^2
\end{aligned}
\end{equation}
where we have used that $\langle \nabla_{E_i}\xi,E_j\rangle^2$ is symmetric in $i,j$ and  $\langle \nabla_{E_i}\xi,E_i\rangle^2=0$ because $\xi$ is a Killing vector field . We now, need to relate $\Vert \nabla_v\xi\Vert$ with the Hilbert-Schmidt norm $\Vert \nabla \xi\Vert$. Recall that for the linear map $\nabla \xi:\mathcal{H}(p)\to\mathcal{H}(p)$ the Hilbert-Schmidt norm is given by
\begin{equation}\label{eq2.5}
\begin{aligned}
\Vert\nabla \xi\Vert^2=&\sum_{i=1}^n\Vert \nabla \xi(E_i)\Vert^2=\sum_{i=1}^n\sum_{j=1}^n\langle \nabla_{E_i}\xi, E_j\rangle^2\\
=&2\sum_{i=1}^{n-1}\sum_{j>i}^n\langle \nabla_{E_i}\xi, E_j\rangle^2
\end{aligned}
\end{equation}
In the particular case when $n=2$, by using  inequalities (\ref{eq2.4}) and (\ref{eq2.5}),
$$
\begin{aligned}
\Vert \nabla_v\xi\Vert^2=&\left((v^1)^2+(v^2)^2\right)\langle\nabla_{E_1}\xi,E_2\rangle^2=\frac{\Vert \nabla\xi\Vert^2}{2}\Vert v\Vert^2 \\
=&\tau^2 \Vert v\Vert^2
\end{aligned}
$$ 
when $n>2$, taking into account that for any $i$ and $j$, $\left((v^i)^2+(v^j)^2\right)\leq \Vert v\Vert^2$,
$$
\begin{aligned}
\Vert \nabla_v\xi\Vert^2\leq& \Vert v\Vert^2 \sum_{i=1}^{n-1}\sum_{j>1}^n\langle \nabla_{E_j}\xi,E_i\rangle^2\\
=&\tau^2 \Vert v\Vert^2
\end{aligned}
$$
and item (2) of the proposition follows. By using ${\rm sec}(v\wedge \xi)=\Vert \nabla_v\xi\Vert^2$ with $\Vert v\Vert^2=1$ item (1) of the proposition follows as well.

Finally, we are going to prove that $\tau$ is an basic function. Given any point $p\in M$ with $\pi(p)=y$ let us consider the integral curve $\gamma_\xi:\mathbb{R}\to M$ of $\xi$ tangent to the fiber $\pi^{-1}(y)$ with $\gamma_\xi(0)=p$ (and $\dot\gamma_\xi(0)=\xi(p)$). It is sufficient to prove  that
$$
\left.\frac{d}{dt}\left(\tau^2\circ\gamma_\xi(t)\right)\right\vert_{t=0}=0
$$
To obtain that let us consider a sufficient small tubular neighborhood of $\gamma((-\epsilon,\epsilon))$ and the following orthonormal basis $\{\xi(\gamma(t)),E_1,\cdots,E_n\}$ at $\gamma(t)$,
($n={\rm dim}(\mathbb{B})$ and $\{E_i\}$ are horizontal vectors). Then
\begin{equation}
\begin{aligned}
\frac{d}{dt}\left(\tau^2\circ\gamma_\xi(t)\right)=&\xi(\tau^2)=\xi\left(\sum_{ij}^n\langle \nabla_{E_i}\xi,E_j\rangle^2\right)\\
=&2\sum_{ij}^n\langle \nabla_{E_i}\xi,E_j\rangle\cdot\xi\left(\langle \nabla_{E_i}\xi,E_j\rangle\right)
\end{aligned}
\end{equation}
But for any $i,j$ 
\begin{equation}
\begin{aligned}
\xi\left(\langle \nabla_{E_i}\xi,E_j\rangle\right)=&\langle \nabla_\xi\nabla_{E_i}\xi,E_j\rangle+\langle \nabla_{E_i}\xi,\nabla_\xi E_j\rangle\\
=& \langle \nabla_{\nabla_{E_i}\xi}\xi,E_j\rangle+\langle\nabla_{E_i}\xi,\nabla_{E_j}\xi\rangle\\
=&-\langle \nabla_{E_j}\xi,\nabla_{E_i}\xi\rangle+\langle\nabla_{E_i}\xi,\nabla_{E_j}\xi\rangle=0
\end{aligned}
\end{equation}
 
\end{proof}
\subsection{Hessian and Laplacian in immersions and submersions}
We are interested in the following setting
\begin{center}
	\begin{tikzcd}
		\Sigma \arrow{r}{\varphi} & M\arrow{r}{\pi}&\mathbb{B}
	\end{tikzcd}
\end{center}
with $\varphi$ an isometric immersion and $\pi$ a Killing submersion. Since in this paper we will assume that $\Sigma$ is stochastically complete, $\Sigma$ satisfies a weak maximum principle for the Laplacian of bounded functions $f:\Sigma\to\mathbb{R}$, see theorem \ref{teopigola}. Our strategy will be to make use of an specific function $\overline{f}: \mathbb{B}\to \mathbb{R}$ and to study the Laplacian of the function $\overline{f}\circ\pi\circ\varphi:\Sigma\to\mathbb{R}$.
In this section, we develop in proposition \ref{propdua} the required relation between $\triangle (\overline{f}\circ\pi\circ\varphi)$, the mean curvature of the immersion $\varphi$ and the bundle curvature $\tau$ of the submersion $\pi$.

Let $\varphi:\Sigma\to M$ be an isometric immersion. For any point $\varphi(p)\in M$ we can decompose the tangent space as $T_{\varphi(p)}M=d\varphi(T_p\Sigma)\oplus (d\varphi(T_p\Sigma))^\perp$.  Let us denote by $\nabla^M$ and $\nabla^\Sigma$  the Levi-Civita connection on $M$ and $\Sigma$. For any $p\in \Sigma$, $x,y\in T_p\Sigma$ and $Y\in \mathfrak{X}(\Sigma)$ an extension of $y$ to $\mathfrak{X}(\Sigma)$, the second fundamental form $II_p(x,y)$ is given by
$$
II_p(x,y)=\nabla^M_{d\varphi(x)}Z-d\varphi(\nabla^\Sigma_xY)
$$
being $Z$ any extension of $d\varphi(Y)$ to  $\mathfrak{X}(M)$.  Since,
$$
(\nabla^M_{d\varphi(x)}Z)^T=d\varphi(\nabla^\Sigma_xY)
$$
the second fundamental form $II_p(x,y)\in (d\varphi(T_p\Sigma)^\perp)$ and recall moreover that the mean curvature of the immersion $\varphi:\Sigma\to M$ in $p$ is defined by
$$
\vec{H}:=\frac{1}{{\rm dim}(\Sigma)}\sum_{i=1}^{{\rm dim}(\Sigma)}II_p(E_i,E_i)
$$
for any orthonormal basis $\{E_i\}$ of $T_p\Sigma$.

Let $f:M\to \mathbb{R}$ be a smooth function, the gradient of $f$ and the gradient of the restricted function $f\circ\varphi:\Sigma\to\mathbb{R}$ satisfy the following relation 
$$
\langle \nabla f\circ \varphi , v\rangle_\Sigma=d(f\circ\varphi)(v)=df(d\varphi(v))=\langle\nabla f,d\varphi(v)\rangle_M.
$$
The Hessian of the restriction $f\circ\varphi$ is given then by
\begin{equation}\label{novaeq10}
\begin{aligned}
{\rm Hess}_\Sigma f\circ\varphi(x,y)=&\langle \nabla^\Sigma_x\nabla f\circ\varphi,y\rangle_\Sigma=x\langle \nabla f\circ\varphi,y\rangle_\Sigma-\langle \nabla f\circ\varphi,\nabla_x^\Sigma Y\rangle_\Sigma\\
=&x\langle \nabla f\circ\varphi,y\rangle_\Sigma-\langle \nabla f,d\varphi(\nabla_x^\Sigma Y)\rangle_M\\
=&d\varphi(x)\langle \nabla f,d\varphi(y)\rangle_M-\langle \nabla f,\nabla^M_{d\varphi(x)}d\varphi(Y)-II_p(x,y)\rangle_M\\
=&{\rm Hess}_Mf(d\varphi(x),d\varphi(y))+\langle \nabla f,II_p(x,y)\rangle_M.
\end{aligned}
\end{equation}
If $\Sigma$ is an  hypersurface of $M$ (\emph{i.e.}, ${\rm dim}(\Sigma)=n$) there exists (at least locally) a vector field $\nu$ normal to $\Sigma$ and such that $(d\varphi(T_p\Sigma))^\perp={\rm span}\{\nu\}$. Given an orthonormal basis $\{E_i\}_{i=1}^n$ of $T_p\Sigma$, $\{d\varphi(E_i)\}_{i=1}^n\cup \{\nu_{\varphi(p)}\}$ is an orthonormal basis of $T_{\varphi(p)}M$ and hence
$$\begin{aligned}
\triangle_\Sigma f\circ \varphi(p)=&\sum_{i=1}^n{\rm Hess}_\Sigma f\circ\varphi(E_i,E_i)=\sum_{i=1}^n{\rm Hess}_M f(d\varphi(E_i),d\varphi(E_i))+n\langle \nabla f,\vec H\rangle_M\\
=&\triangle_M f(\varphi(p))-{\rm Hess}_Mf(\nu_{\varphi(p)},\nu_{\varphi(p)})+n\langle \nabla f,\vec H\rangle_M
\end{aligned}
$$

But now we are interested in the particular case when $f:M\to\mathbb{R}$ is the lift of a basic function $\overline f:\mathbb{B}\to \mathbb{R}$, namely $f=\overline f\circ \pi$. 
\begin{proposition}\label{propdua}Let $\Sigma$ be an hypersurface immersed in $M$ by  $\varphi:\Sigma\to M$, let $M$ admit a Killing submersion $\pi: M\to \mathbb{B}$ with unit-length Killing vector field $\xi\in\mathfrak{X}(M)$. Let $\overline f:\mathbb{B}\to\mathbb{R}$ be a smooth function on the base manifold. Denote by $f=\overline f\circ \pi$ the lift of $\overline f$. Then,
\begin{equation}\label{unua}
\triangle_\Sigma f\circ \varphi(p)=\triangle_\mathbb{B} \overline f(\pi\circ\varphi(p))-{\rm Hess}_B\overline f(d\pi(\nu),d\pi(\nu))+n\langle \nabla f,\vec H_\tau\rangle
\end{equation}
where 
\begin{equation}\label{bundlemeancurvature}
\vec H_\tau:=\vec H+\frac{2}{n}\langle \nu,\xi\rangle \nabla\xi(\nu_H)
\end{equation}
\end{proposition}
\begin{proof}
If $f:M\to\mathbb{R}$ is the lift of a basic function $\overline f:\mathbb{B}\to \mathbb{R}$, 
$$
{\rm Hess}_Mf(X,Y)=\langle X,\nabla_Y\nabla f\rangle_M=Y\langle X,\nabla f\rangle_M-\langle\nabla_YX,\nabla f\rangle_M
$$
Observe that since 
$$
\langle \nabla f,X\rangle_M=df(X)=d(\overline f\circ \pi)(X)=d\overline f(d\pi(X))=\langle \nabla \overline f,d\pi (X)\rangle_\mathbb{B}
$$
then $\nabla f$ is an horizontal vector field in $\mathfrak{X}(M)$, $\pi$-related with $\nabla\overline f\in \mathfrak{X}(\mathbb{B})$. Let us decompose $\nu=\nu_H+\nu_V$ in its horizontal and vertical part, then
$$
\begin{aligned}
{\rm Hess}_Mf(\nu,\nu)=&{\rm Hess}_Mf(\nu_H,\nu_H)+2{\rm Hess}_Mf(\nu_H,\nu_V)+{\rm Hess}_Mf(\nu_V,\nu_V)\\
=&{\rm Hess}_Mf(\nu_H,\nu_H)+2\langle\nu,\xi\rangle{\rm Hess}_Mf(\nu_H,\xi)+\langle\nu,\xi\rangle^2{\rm Hess}_Mf(\xi,\xi)\\
=&{\rm Hess}_Mf(\nu_H,\nu_H)-2\langle\nu,\xi\rangle\langle\nabla_{\nu_H}\xi,\nabla f\rangle-\langle\nu,\xi\rangle^2\langle\nabla_\xi\xi,\nabla f\rangle\\
=&{\rm Hess}_Mf(\nu_H,\nu_H)-2\langle\nu,\xi\rangle\langle\nabla\xi(\nu_H),\nabla f\rangle\\
=&\nu_H(\langle \nu_H,\nabla f\rangle)-\langle \nabla_{\nu_H}\nu_H,\nabla f\rangle-2\langle\nu,\xi\rangle\langle\nabla\xi(\nu_H),\nabla f\rangle\\
=&d\pi(\nu)(\langle d\pi(\nu),\nabla \overline f\rangle_\mathbb{B})-\langle \nabla_{d\pi(\nu)}^\mathbb{B}d\pi(\nu),\nabla \overline f\rangle_\mathbb{B}-2\langle\nu,\xi\rangle\langle\nabla\xi(\nu_H),\nabla f\rangle\\
=&{\rm Hess}_\mathbb{B}\overline f(d\pi(\nu),d\pi(\nu))-2\langle\nu,\xi\rangle\langle\nabla\xi(\nu_H),\nabla f\rangle
\end{aligned}
$$
where we have used that $d\pi(\nabla_{\nu_H}\nu_H)=\nabla_{d\pi(\nu)}^\mathbb{B}d\pi(\nu)$ see \cite{OneillSubmersion}. Therefore,
$$
\begin{aligned}
\triangle_\Sigma f\circ \varphi(p)=&\triangle_M f(\varphi(p))-{\rm Hess}_\mathbb{B}\overline f(d\pi(\nu),d\pi(\nu))+2\langle\nu,\xi\rangle\langle\nabla\xi(\nu_H),\nabla f\rangle\\ &+n\langle \nabla f,\vec H\rangle_M
\end{aligned}
$$
Moreover given the orthonormal basis $\{E_i\}_{i=1}^n\cup \{\xi\}$ (with $E_i$ horizontals),
$$
\begin{aligned}
\triangle_Mf(\varphi(p))=&\sum_{i=1}^n{\rm Hess}_Mf(E_i,E_i)+{\rm Hess}_Mf(\xi,\xi)\\
=&\sum_{i=1}^n{\rm Hess}_Mf(E_i,E_i)=\sum_{i=1}^n{\rm Hess}_\mathbb{B}\overline f(d\pi(E_i),d\pi(E_i))=\triangle_\mathbb{B}\overline f(\pi(\varphi(p)).
\end{aligned}
$$
Hence, finally
$$
\begin{aligned}
\triangle_\Sigma f\circ \varphi(p)=&\triangle_\mathbb{B} \overline f(\pi\circ\varphi(p))-{\rm Hess}_\mathbb{B}\overline f(d\pi(\nu),d\pi(\nu))+2\langle\nu,\xi\rangle\langle\nabla\xi(\nu_H),\nabla f\rangle\\&+n\langle \nabla f,\vec H\rangle
\end{aligned}
$$

In order to simply the expression we will make us of  the $\tau$-mean curvature $\vec H_\tau$ of $\Sigma$ defined in (\ref{bundlemeancurvature}). Then
\begin{equation}
\triangle_\Sigma f\circ \varphi(p)=\triangle_\mathbb{B} \overline f(\pi\circ\varphi(p))-{\rm Hess}_\mathbb{B}\overline f(d\pi(\nu),d\pi(\nu))+n\langle \nabla f,\vec H_\tau\rangle
\end{equation}
\end{proof}
\subsection{Radial functions on the base manifold}

Suppose that $\overline f :\mathbb{B}\to \mathbb{R}$ is a radial function with respect to the point $o\in \mathbb{B}$, in the sense that $\overline{f}(x)=\overline{f}(y)$ if $r_o(x)={\rm dist}_\mathbb{B}(o,x)=r_o(y)$, then there exists a function $F:\mathbb{R}\to\mathbb{R}$ such that 
$$
\overline f(x)=F\circ r_o (x)
$$
for any $x\in \mathbb{B}$.  Now, in the following proposition we will obtain bounds on the Hessian and Laplacian of $\overline f$  
\begin{proposition}\label{proptria}Let $\mathbb{B}$ a Riemannian manifold, let $o\in \mathbb{B}$, and denote by $r_o:\mathbb{B}\to \mathbb{R}$ the distance function in $\mathbb{B}$ to $o$, \emph{i.e.}, $r_o(p)={\rm dist}_\mathbb{B}(o,p)$. Assume moreover that the sectional curvatures of $\mathbb{B}$ are bounded from above and below for any plane of the tangent space,
$$
k\leq {\rm sec}(\mathbb{B})\leq \kappa.
$$
Then, for any function $F:\to \mathbb{R}\to\mathbb{R}$ with $F'\geq 0$,
\begin{equation}\label{dua}\begin{aligned}
 -{\rm Hess}_\mathbb{B}\overline f_x(X,X)\geq&-F''(t)\langle \nabla r_o,X\rangle^2-F'(t)\frac{{\rm sn}_k'(t)}{{\rm sn}_k(t)}\left(\Vert X\Vert^2-\langle X,\nabla r_o\rangle^2\right)\\
 \triangle_\mathbb{B}F\geq & F''(t)+(n-1)F'(t)\frac{{\rm sn}_K'(t)}{{\rm sn}_K(t)}
\end{aligned}
\end{equation}
where here $\overline f=F\circ r_o$ , $t=r_0(x)$ and
$$
{\rm sn}_K(t):=\left\{\begin{array}{lcl}
\frac{\sin(\sqrt{K}t)}{\sqrt{K}}&{\rm if}& K>0\\
t&{\rm if}&K=0\\
\frac{\sinh(\sqrt{-K}t)}{\sqrt{-K}}&{\rm if}& K<0
\end{array}\right.
$$
\end{proposition}
\begin{proof}
By using the definition of the Hessian and the chain rule,
$$\begin{aligned}
{\rm Hess}_\mathbb{B}\overline f_x(X,X)=&\langle \nabla_X\nabla f,X\rangle=\langle \nabla_XF'\nabla r_0,X\rangle\\= &F''(t)\langle \nabla r_0,X\rangle^2+F'(t)\langle \nabla_X\nabla r_0,X\rangle\\
=&F''(t)\langle \nabla r_0,X\rangle^2+F'(t){\rm Hess}_\mathbb{B}r_o(X,X).
\end{aligned}
$$
Therefore,
$$
\triangle_\mathbb{B}\overline f(x)=F''(t)+F'(t)\triangle_\mathbb{B}r_o(x).
$$
But if the sectional curvatures of the base manifold are bounded as $k\leq{\rm sec}\leq \kappa$, see Theorem 27 of \cite{Petersen},
$$
\frac{{\rm sn}_\kappa'(r_o(x))}{{\rm sn}_\kappa(r_o(x))}\left(\Vert X\Vert^2-\langle X,\nabla r_o\rangle^2\right)\leq {\rm Hess}_\mathbb{B}r_o(X,X)
$$
and
$$
{\rm Hess}_\mathbb{B}r_o(X,X)\leq \frac{{\rm sn}_k'(r_o(x))}{{\rm sn}_k(r_o(x))}\left(\Vert X\Vert^2-\langle X,\nabla r_o\rangle^2\right)
$$
Then,
$$
\frac{{\rm sn}_\kappa'(r_o(x))}{{\rm sn}_\kappa(r_o(x))}(n-1)\leq\triangle_\mathbb{B} r_0\leq \frac{{\rm sn}_k'(r_o(x))}{{\rm sn}_k(r_o(x))}(n-1)
$$

Hence, finally if $F'>0$,
\begin{equation}\begin{aligned}
 -{\rm Hess}_\mathbb{B}\overline f_x(X,X)\geq&-F''(t)\langle \nabla r_o,X\rangle^2-F'(t)\frac{{\rm sn}_k'(t)}{{\rm sn}_k(t)}\left(\Vert X\Vert^2-\langle X,\nabla r_o\rangle^2\right)\\
 \triangle_\mathbb{B}F\geq & F''(t)+(n-1)F'(t)\frac{{\rm sn}_\kappa'(t)}{{\rm sn}_\kappa(t)}
\end{aligned}
\end{equation}
\end{proof}
If we have a Killing submersion $\pi:M\to \mathbb{B}$ we can lift the radial function $\overline f$ to $\widetilde f=\overline f\circ \pi$  and using equation (\ref{unua}) of proposition \ref{propdua} we obtain for $F'>0$,
\begin{equation}\label{tria}\begin{aligned}
\triangle_\Sigma f(z)\geq & F''(t)\left(1-\langle \nabla r_o,d\pi(\nu)\rangle^2\right)+(n-1)F'(t)\frac{{\rm sn}_\kappa'(t)}{{\rm sn}_\kappa(t)}\\ &-F'(t)\frac{{\rm sn}_k'(t)}{{\rm sn}_k(t)}\left(\Vert d\pi(\nu)\Vert^2-\langle d\pi(\nu),\nabla r_o\rangle^2\right)\\ & +nF'(t)\langle \nabla r_o,\vec H_\tau\rangle\end{aligned}
\end{equation}
where $f=\widetilde f\circ\varphi$ and $t=r_0(\pi\circ\varphi(z))$. This above inequality can be rewritten in the following corollary,

\begin{corollary}\label{corunua}Let $\Sigma$ be an hypersurface immersed in $M$ by  $\varphi:\Sigma\to M$, let $M$ admit a Killing submersion $\pi: M\to \mathbb{B}$ with unit-length Killing vector field. Suppose that the sectional curvatures of $\mathbb{B}$ are bounded from above and below by
$$
k\leq {\rm sec}(B)\leq \kappa.
$$
Let $F:\mathbb{R}\to \mathbb{R}$ be a smooth and non-decreasing function, let $r_o:B\to \mathbb{R}$ be the distance function in $B$ to $o\in\mathbb{B}$, \emph{i.e.}, $r_o(p)={\rm dist}_B(o,p)$, denote by $f=F\circ r_o\circ \pi$. Then,
\begin{equation}\label{tria2}\begin{aligned}
\triangle_\Sigma f(z)\geq & F''(t)+(n-1)F'(t)\frac{{\rm sn}_\kappa'(t)}{{\rm sn}_\kappa(t)}-F'(t)\frac{{\rm sn}_k'(t)}{{\rm sn}_k(t)}\Vert d\pi(\nu)\Vert^2\\&+\left(F'(t)\frac{{\rm sn}_k'(t)}{{\rm sn}_k(t)}-F''(t)\right)\langle \nabla r_o,d\pi(\nu)\rangle^2\\ & +nF'(t)\langle \nabla r_o,\vec H_\tau\rangle\end{aligned}
\end{equation}

where $t=r_o(\pi(\varphi(z)))$ and $\vec H_\tau$ is given by definition (\ref{bundlemeancurvature}).
\end{corollary}

\subsection{Stochastic Completeness, weak maximum principle and Omori-Yau maximum principle}\label{stochastic}
Let $\Sigma$ be a complete and non compact Riemannian manifold. The heat kernel of $\Sigma$ is a function $p_t(x,y)$ on $(0,\infty)\times \Sigma\times \Sigma$ which is the minimal positive fundamental solution to the heat equation
$$
\frac{\partial v}{\partial t}=\triangle v.
$$
In other words, the Cauchy problem 
$$
\left\{
\begin{aligned}
&\frac{\partial v}{\partial t}=\triangle v\\
&v\vert_{t=0}=v_0(x)
\end{aligned}\right.
$$
has a solution 
$$
v(x,t)=\int_\Sigma p_t(x,y)v_0(y)dV(y)
$$
provided that $v_0$ is a bounded continuous positive function. The manifold $\Sigma$ is said to be stochastically complete, see \cite{GriExp}, if
$$
\int_\Sigma p_t(x,y)dV(y)=1
$$
for any $x\in \Sigma$ and any $t>0$.
The main property of stochastic completeness which is used in this paper is that if a Riemannian manifold is stochastic complete a weak maximum principle is satisfied for bounded functions in $C^2$. 
More precisely, if $\Sigma $ is stochastically complete we can state the following theorem

\begin{theorem}[See \cite{Pigola2003}]\label{teopigola}Let $\Sigma$ be a connected non-compact Riemannian manifold. Suppose that $\Sigma$ is stochastically complete, then for every $u\in C^2(\Sigma)$ with $\sup_\Sigma u<\infty$ there exists a sequence $\{x_k\}$, $k=1,2,\ldots$, such that, for every $k$, $u(x_k)\geq \sup_\Sigma u-1/k$ and $\triangle u(x_k)\leq 1/k$.
\end{theorem}

On the other hand (see \cite{Alias2016}), a Riemannian manifold $(M,g)$ satisfies the \emph{Omori-Yau maximum principle for the Laplacian} if for any function $u\in C^2(M)$ which is bounded $\sup_M u=u^*<\infty$, there exists a sequence $\{x_i\}_{i\in \mathbb{N}}\subset M$ such that 
$$
u(x_i)>u^*-\frac{1}{i},\quad \Vert \nabla u(x_i)\Vert<\frac{1}{i},\quad \triangle u(x_i)<\frac{1}{i}.
$$
In this paper we will use the following sufficient condition for the Omori-Yau maximum principle

\begin{theorem}[See \cite{Alias2016}]\label{AliasTeo}Let $\Sigma$ be a connected non-compact Riemannian manifold. Suppose that $\Sigma$ admits a $C^2$ function $f:\Sigma\to\mathbb{R}$ satisfying 
\begin{enumerate}
\item $f(x)\to\infty$ when $x\to\infty$
\item $\Vert \nabla f\Vert \leq G(f)$ outside a compact subset of $\Sigma$.
\item $\triangle f\leq G(f)$  outside a compact subset of $\Sigma$.
\end{enumerate}
with $G\in C^1(\mathbb{R}_+)$, positive near infinity and such that
$$
\frac{1}{G}\notin L^1(+\infty)\quad {\rm and}\quad G'(t)\geq -A(\log(t)+1)
$$
for $t$ large enough and $A\geq 0$.  Then, the Omori-Yau maximum principle for the Laplacian holds on $\Sigma$.
\end{theorem}

\section{Proof of theorem \ref{maintheo} }\label{sec4}
The statement of the theorem \ref{maintheo}  is as follows
\begin{theorem*}
Let $\Sigma$ be a complete and non-compact Riemannian manifold. Let $\varphi:\Sigma\to M$ be an isometric immersion. Suppose that $M$ admits a Killing submersion $\pi:M\to \mathbb{B}$ with unit-length Killing vector field, suppose moreover that $\varphi(\Sigma)\subset \pi^{-1}(B_R^{\mathbb{B}}(p))$ for some geodesic ball $B_R^{\mathbb{B}}(o)$ of radius $R$ centered at $o\in \mathbb{B}$. Assume that the sectional curvatures are bounded ${\rm sec }\leq \kappa$ in $B_R^{\mathbb{B}}(o)$ and that
$$
R<\min\left\{{\rm inj}(o),\,\frac{\pi}{2\sqrt{\kappa}}\right\}
$$
where ${\rm inj}(o)$ is the injectivity radius of $o$ and we replace $\pi/2\sqrt{\kappa}$ by $+\infty$ if $\kappa<0$. Then, if $\Sigma$ is stochastically complete, the supremum of the norm of the mean curvature vector field of $\Sigma$ satisfies 
$$
\sup_\Sigma\Vert \vec H\Vert\geq h_{\kappa}^n(R)-\frac{\tau_\Sigma}{n}
$$
where $h_{\kappa}^n(R)$ is the norm of the mean curvature of the generalized cylinder $\partial B^{\mathbb{M}^n(\kappa)}_R\times \mathbb{R}$ in $\mathbb{M}^n(\kappa)\times \mathbb{R}$.
\end{theorem*}

\begin{proof}
Since $\pi(\varphi(\Sigma))$ is bounded and contained in the geodesic ball $B_R^\mathbb{B}(o)$ for some $o\in \mathbb{B}$
$$
R_*:=\sup_{\pi(\varphi(\Sigma))} r_o\leq R<\infty,\quad {\rm with }\quad r_o(\cdot)={\rm dist}_\mathbb{B}(o,\cdot).
$$
Moreover, for any $2$-plane $\Pi_p\subset T_p\mathbb{B}$ the sectional curvatures ${\rm sec}(\Pi_p)$ of any $p\in B_R^\mathbb{B}(o)$ will be bounded as
$$
-\infty<k\leq \inf_{p\in B_R^\mathbb{B}(o)}{\rm sec}(\Pi_p)\leq {\rm sec}(\Pi_p)\leq \kappa.
$$
In order to simplify the argument of the proof let us choose $k<0$, and let us define the function
$$
\begin{aligned}
&F_k:\mathbb{R}\to\mathbb{R},\quad t\to F_k(t)=\int_0^t{\rm sn}_k(s)ds
\end{aligned}
$$
Now we are going to compute the Laplacian of $f=F_k\circ r_o\circ\pi\circ\varphi$. By using inequality (\ref{tria2}) of corollary \ref{corunua},
\begin{equation}\label{kvara}\begin{aligned}
\triangle_\Sigma f(z)\geq & {\rm sn}'_k(t(z))+(n-1){\rm sn}_k(t(z))\frac{{\rm sn}_\kappa'(t(z))}{{\rm sn}_\kappa(t(z))}-{\rm sn}_k'(t(z))\Vert d\pi(\nu)\Vert^2\\& -n\, {\rm sn}_k(t(z))\Vert \vec H_\tau(z)\Vert \end{aligned}
\end{equation}
Since $k\leq 0$, then ${\rm sn}_k'\geq 0$ and hence
\begin{equation}\label{kvina}\begin{aligned}
\triangle_\Sigma f(z)\geq & (n-1){\rm sn}_k(t(z))\frac{{\rm sn}_\kappa'(t(z))}{{\rm sn}_\kappa(t(z))}-n\, {\rm sn}_k(t(z))\Vert \vec H_\tau(z)\Vert\end{aligned}
\end{equation}

Now, we are going to apply theorem \ref{teopigola} to $f:\Sigma\to \mathbb{R}$ because since $\displaystyle\sup_{\Sigma}r_0\circ \pi=R_*$, and $F_k$ is an increasing function, $\displaystyle\sup_{\Sigma}f=\int_0^{R_*}{\rm sn}_k(s)ds<\infty$. Then, there exists a sequence $\{x_i\}$, such that
$$
f(x_i)\geq \sup_\Sigma f-\frac{1}{i}, \quad \text{and}\quad \triangle f(x_i)\leq \frac{1}{i}
$$
Therefore $t(x_i)\to R_*$ when $i\to \infty$, and by inequality (\ref{kvina}), 

\begin{equation}\label{sesa}\begin{aligned}
\frac{1}{i}\geq & (n-1){\rm sn}_k(t(x_i))\frac{{\rm sn}_\kappa'(t(x_i))}{{\rm sn}_\kappa(t(x_i))}-n\, {\rm sn}_k(t(x_i))\Vert \vec H_\tau(x_i)\Vert\end{aligned}
\end{equation}
Then,
\begin{equation}\label{sepa}\begin{aligned}
\Vert \vec H_\tau(x_i)\Vert\geq & \frac{(n-1)}{n}\frac{{\rm sn}_\kappa'(t(x_i))}{{\rm sn}_\kappa(t(x_i))}-\frac{1}{n\, {\rm sn}_k(t(x_i))i}\end{aligned}
\end{equation}
But $\Vert \vec{H}_\tau\Vert=\Vert \vec{H}+\frac{2}{n}\langle \nu,\xi\rangle \nabla\xi(\nu_H)\Vert\leq\Vert \vec{H}\Vert+\frac{2}{n}\Vert\langle \nu,\xi\rangle\Vert \Vert\nabla\xi(\nu_H)\Vert $. Hence, denoting $\theta=\arccos(\langle \nu,\xi\rangle)$ and applying proposition \ref{propunua},
\begin{equation}\label{deka}
\begin{aligned}
\Vert \vec{H}_\tau\Vert\leq&\Vert \vec{H}\Vert+\frac{2}{n}\vert    cos(\theta)\vert\, \Vert\nu_H\Vert\ \tau=\Vert \vec{H}\Vert+\frac{\tau}{n}\vert    \sin(2\theta)\vert\\
\leq& \Vert \vec{H}\Vert+\frac{\tau_\Sigma}{n}
\end{aligned}
\end{equation}
Therefore,
\begin{equation}\label{oka}\begin{aligned}
\sup_\Sigma\Vert \vec H\Vert\geq & \frac{(n-1)}{n}\frac{{\rm sn}_\kappa'(t(x_i))}{{\rm sn}_\kappa(t(x_i))}-\frac{1}{n\, {\rm sn}_k(t(x_i))i}-\frac{\tau_\Sigma}{n}
\end{aligned}
\end{equation}
Letting now $i$ tend to infinity,
\begin{equation}\label{nona}\begin{aligned}
\sup_\Sigma\Vert \vec H\Vert\geq & \frac{(n-1)}{n}\frac{{\rm sn}_\kappa'(R_*)}{{\rm sn}_\kappa(R_*)}-\frac{\tau_\Sigma}{n}\cdot\end{aligned}
\end{equation}
Since $R_*\leq  R<\pi/2\sqrt{\kappa}$, $\frac{{\rm sn}_{\kappa}'}{{\rm sn}_{\kappa}}$ is a decreasing function and therefore
\begin{equation}\label{dekunua}\begin{aligned}
\sup_\Sigma\Vert \vec H\Vert\geq & \frac{(n-1)}{n}\frac{{\rm sn}_\kappa'(R)}{{\rm sn}_\kappa(R)}-\frac{\tau_\Sigma}{n}\cdot\end{aligned}
\end{equation}
Finally the theorem follows by taking into account that $\frac{(n-1)}{n}\frac{{\rm sn}_\kappa'(R)}{{\rm sn}_\kappa(R)}$ is the norm of the mean curvature of the generalized cylinder $\partial B_R^{\mathbb{M}^n(\kappa)}\times\mathbb{R}$ in $\mathbb{M}^n(\kappa)\times\mathbb{R}$.
\end{proof}
\begin{remark}
In the proof of theorem \ref{maintheo} we have used in inequality (\ref{oka}) that $\tau\leq \tau_\Sigma$ and $\vert \sin(2\theta)\vert\leq 1$.  We could use instead the following factors to improve the result
$$
\begin{aligned}
&\tau^*:=\lim_{\rho\to R_*}\sup\left\{\tau(x)\,:\,x\in\Sigma\setminus \pi^{-1}(B_\rho(p))\right\}\\
& \alpha:=\lim_{\rho\to R_*}\sup\left\{\vert \sin(2\theta(x))\vert \,:\,x\in\Sigma\setminus \pi^{-1}(B_\rho(p))\right\}
\end{aligned}
$$
and the inequality would be
\begin{equation}\begin{aligned}
\sup_\Sigma\Vert \vec H\Vert\geq & \frac{(n-1)}{n}\frac{{\rm sn}_\kappa'(R)}{{\rm sn}_\kappa(R)}-\frac{\tau^*\alpha}{n}\cdot\end{aligned}
\end{equation}
 \end{remark}

\section{Proof of theorem \ref{teo2}}\label{sec5}

In the statement of theorem \ref{teo2} it is assumed that the Killing submersion $\pi:M\to \mathbb{B}$ admits a smooth section $s:\overline{B_R^\mathbb{B}(o)}\to M$ and is assumed as well that the normal exponential map $$\exp:s(B_R^\mathbb{B}(o))\times \mathbb{R}\to \pi^{-1}((B_R^\mathbb{B}(o))), \quad (p,z)\to \exp(p,z)=\exp_p(z\xi)$$ is a diffeomorphism.

Given the section $s:\overline{B_R^\mathbb{B}(o)}\to M$ we can trivialize $\pi^{-1}(B_R^\mathbb{B}(o))\approx s(B_R^\mathbb{B}(o))\times \mathbb{R}$ using the following map $T:\pi^{-1}(B_R^\mathbb{B}(o))\approx s(B_R^\mathbb{B}(o))\times \mathbb{R}$ given by
$$
q\in \pi^{-1}(B_R^\mathbb{B}(o))\to T(q)=(p(q),z(q))
$$
where $p$ and $z$ are the following two functions 
\begin{equation}\label{nova24}
\begin{aligned}
&p:\pi^{-1}(B_R^\mathbb{B}(o))\to s(B_R^\mathbb{B}(o)),\quad p(q):=s(\pi(q))\\
&z:\pi^{-1}(B_R^\mathbb{B}(o))\to \mathbb{R},\quad q=\exp_{p(q)}(z(q)\xi)
\end{aligned}
\end{equation}
Observe that we can define the $z$ function by an implicit equation because the hypothesis on the injectivity of normal exponential map. Observe moreover that
$$
z(q)\geq {\rm dist}^M(q,s(B_R^\mathbb{B}(o))
$$
because the curve $t\to \gamma(t)=\exp_{p(q)}(t\xi)$ is a geodesic joining $\gamma(0)=p(q)\in s(B_R^\mathbb{B}(o))$ with $q=\gamma(z(q))$. 
We can can moreover reverse the $T$  map 
$$
T^{-1}:s(B_R^\mathbb{B}(o))\times\mathbb{R}\to\pi^{-1}(B_R^\mathbb{B}(o)),\quad (p,z)\to T^{-1}(p,z)=\exp_{p}(z\xi).
$$
We will need furthermore the expression of the gradient of the $z$ function which is given in the following lemma.
\begin{lemma}Let $\pi: M\to \mathbb{B}$ a Killing submersion with unit-length Killing vector field $\xi$. Suppose that the Killing submersion $\pi:M\to \mathbb{B}(o)$ admits a smooth section $s:\overline{B_R^\mathbb{B}(o)}\to M$ and that the normal exponential map $$\exp:s(B_R^\mathbb{B}(o))\times \mathbb{R}\to \pi^{-1}(B_R^\mathbb{B}(o)), \quad (p,z)\to \exp(p,z)=\exp_p(z\xi)$$ is a diffeomorphism. Then the gradient of the map $z:\pi^{-1}(B_R^\mathbb{B}(o))\to \mathbb{R}$ defined in equation (\ref{nova24})
is given by $$\nabla z=\xi.$$\end{lemma}
\begin{proof}
For any $q\in \pi^{-1}(B_R^\mathbb{B}(o))$, 
$$
T_qM=dT^{-1}(T_{p(q)}s(B_R^{\mathbb{B}}(o)))+dT^{-1}(T_{z(q)}\mathbb{R})
$$ 
Observe moreover that for any $v \in T_{p(q)}s(B_R^{\mathbb{B}}(o))$,
$$
dz(dT^{-1}(v))=\frac{d}{dt}z(\exp_{\gamma(t)}(z(q)\xi))\vert_{t=0}=\frac{d}{dt}z(q)=0
$$
for any curve $\gamma:(-\epsilon,\epsilon)\to s(B_R^\mathbb{B}(o))$ with $\gamma(0)=p$ and $\dot\gamma(0)=v$. Then,
$$
dz(dT^{-1}(v))=\langle\nabla z,dT^{-1}(v)\rangle=0.
$$
This implies that 
$
\nabla z(q)\in dT^{-1}(T_{z(q)}\mathbb{R}).
$
Then
$$
\nabla z(q)=\alpha(q)\xi(q)
$$
but taking now the curve $\gamma_\xi(t)=\exp_{p(q)}((z(q)+t)\xi)$ which is the integral curve of $\xi$ with $\gamma_\xi(o)=q$ and $\dot\gamma_\xi(0)=\xi(q)$,
$$
\frac{d}{dt}z(\gamma_\xi(t))\vert_{t=0}=\frac{d}{dt}t\vert_{t=0}=1=dz(\xi)=\langle\nabla z(q),\xi(q)\rangle
$$
we conclude that
$$
\nabla z(q)=\xi(q).
$$\end{proof}
 
 \begin{proposition}\label{novaprop5.2}
  Let $\Sigma$ be a complete and non-compact Riemannian manifold. Let $\varphi:\Sigma\to M$ be a proper isometric immersion. Suppose that $M$ admits a Killing submersion $\pi:M\to \mathbb{B}$ with unit-length Killing vector field, suppose moreover that $\varphi(\Sigma)\subset \pi^{-1}(B_R^{\mathbb{B}}(o))$ for some geodesic ball $B_R^{\mathbb{B}}(o)$ of radius $R$ centered at $o\in \mathbb{B}$. Assume that the sectional curvatures are bounded ${\rm sec }\leq \kappa$ in $B_R^{\mathbb{B}}(o)$ and that
$$
R<\min\left\{{\rm inj}(o),\,\frac{\pi}{2\sqrt{\kappa}}\right\}
$$
where ${\rm inj}(o)$ is the injectivity radius of $o$ and we replace $\pi/2\sqrt{\kappa}$ by $+\infty$ if $\kappa<0$.  Suppose moreover, that $\pi$ admits a smooth section $s:\overline{B_R^\mathbb{B}(o)}\to M$ and the normal exponential map
$$
\exp: s(B_R^\mathbb{B}(o))\times \mathbb{R}\to \pi^{-1}((B_R^\mathbb{B}(o))),\quad (p,z)\to \exp(p,z):=\exp_p(z\xi)
$$
is a diffeomorphism. If moreover,
$$
\sup_{\Sigma}\Vert\vec{H}(p)\Vert<\infty
$$
Then, the Omori-Yau principle for the Laplacian holds in $\Sigma$.
\end{proposition}
 \begin{proof}

In order to prove that the Omori-Yau principle for the Laplacian holds in $\Sigma$ we are using theorem \ref{AliasTeo} with the function $f=z\circ \varphi$.  Hence, we only need to prove that the function $f=z\circ \varphi$ when $\varphi$ is proper, and the hypothesis of the proposition are fulfilled, satisfies 
\begin{enumerate}
\item $f(x)\to \infty$ when $x\to \infty$.
\item $\vert \nabla f\vert\leq G(f)$.
\item $\triangle_\Sigma f\leq G(f)$ outside of a compact set.
\end{enumerate}
with $G(t)=\sqrt{t^2+1}$.

Observe that since the immersion is proper, and $\Sigma$ is a complete and non-compact manifold, if ${\rm dist}^\Sigma(x,x_0)\to\infty$ then 
${\rm dist}^M(\varphi(x),\varphi(x_0))\to\infty$. But for any $p\in s(B_R^\mathbb{B}(o))$
$$
{\rm dist}^M(\varphi(x),\varphi(x_0))\leq {\rm dist}^M(\varphi(x),p)+{\rm dist}^M(p,\varphi(x_0))
$$
Therefore, since $\overline{B_R^\mathbb{B}(o)}$ is compact there exists $p_*\in s(\overline{B_R^\mathbb{B}(o)})$ such that the distance is minimized at $p_*$, \emph{i.e.}, ${\rm dist}^M(\varphi(x),p_*)={\rm dist}^M(\varphi(x),s(\overline{B_R^\mathbb{B}(o)}))$ and an other $p^*\in s(\overline{B_R^\mathbb{B}(o)})$ such that ${\rm dist}^M(p,\varphi(x_0))\leq {\rm dist}^M(p^*,\varphi(x_0))$ for any $p\in s(\overline{B_R^\mathbb{B}})$.  Hence,
$$
\begin{aligned}
{\rm dist}^M(\varphi(x),\varphi(x_0))\leq &{\rm dist}^M(\varphi(x),s(\overline{B_R^\mathbb{B}}))+{\rm dist}^M(p^*,\varphi(x_0))\\
\leq &z\circ\varphi(x)+{\rm dist}^M(p^*,\varphi(x_0))\\
\leq & f(x)+{\rm dist}^M(p^*,\varphi(x_0))
\end{aligned}
$$
 
then, $f(x)\to \infty$ when $x\to\infty$ as  it was stated.
On the other hand,
$$
\langle \nabla f,v\rangle=\langle \nabla z,d\varphi(v)\rangle
$$
then,
$$
\Vert \nabla f\Vert^2\leq \langle \nabla f,\nabla f\rangle=\langle \nabla z,d\varphi(\nabla f)\rangle=\Vert \nabla^Tz\Vert^2\leq \Vert \nabla z\Vert^2\leq \Vert \xi\Vert^2=1\leq G(f(z))
$$
Now we are going to compute $\triangle f$. By equation (\ref{novaeq10}), for any orthonormal basis $\{E_i\}$ of $T_p\Sigma$
$$
\begin{aligned}
{\rm Hess}_{\Sigma}f(E_i,E_i)=&{\rm Hess}_{M}z(d\varphi(E_i),d\varphi(E_i))+\langle \nabla z,II_p(E_i,E_i)\rangle\\
=&\langle \nabla_{d\varphi(E_i)}\nabla z,d\varphi(E_i)\rangle+\langle \nabla z,II_p(E_i,E_i)\rangle\\
=&\langle \nabla_{d\varphi(E_i)}\xi,d\varphi(E_i)\rangle+\langle \xi,II_p(E_i,E_i)\rangle\\
=&\langle \xi,II_p(E_i,E_i)\rangle
\end{aligned}
$$ 
Then,
$$
\triangle_\Sigma f=\sum_{i=1}^n{\rm Hess}_\Sigma f(E_i,E_i)=n\langle \xi,\vec{H}(p)\rangle
$$
And hence,
$$
\triangle_\Sigma f\leq n\sup_{\Sigma}\Vert\vec{H}(p)\Vert
$$
Then if $\sup_{\Sigma}\Vert\vec{H}(p)\Vert<\infty$, since $f(x)\to \infty$ when $x\to\infty$,
$$
\triangle_\Sigma f\leq G(f)
$$
outside of a compact set, and the proposition follows.
 \end{proof}
 The statement of theorem \ref{teo2} and its proof is as follows
 \begin{theorem*}
 Let $\Sigma$ be a complete and non-compact Riemannian manifold. Let $\varphi:\Sigma\to M$ be a proper isometric immersion. Suppose that $M$ admits a Killing submersion $\pi:M\to \mathbb{B}$ with unit-length Killing vector field, suppose moreover that $\varphi(\Sigma)\subset \pi^{-1}(B_R^{\mathbb{B}}(p))$ for some geodesic ball $B_R^{\mathbb{B}}(o)$ of radius $R$ centered at $o\in \mathbb{B}$. Assume that the sectional curvatures are bounded ${\rm sec }\leq \kappa$ in $B_R^{\mathbb{B}}(o)$ and that
$$
R<\min\left\{{\rm inj}(o),\,\frac{\pi}{2\sqrt{\kappa}}\right\}
$$
where ${\rm inj}(o)$ is the injectivity radius of $o$ and we replace $\pi/2\sqrt{\kappa}$ by $+\infty$ if $\kappa<0$.  Suppose moreover, that $\pi$ admits a smooth section $s:\overline{B_R^\mathbb{B}(o)}\to M$ and the normal exponential map
$$
\exp: s(B_R^\mathbb{B}(o))\times \mathbb{R}\to \pi^{-1}((B_R^\mathbb{B}(o))),\quad (p,z)\to \exp(p,z):=\exp_p(z\xi)
$$
is a diffeomorphism.  Then,  the supremum of the norm of the mean curvature vector field of $\Sigma$ satisfies 
$$
\sup_\Sigma\Vert \vec H\Vert\geq h_{\kappa}^n(R)
$$
where $h_{\kappa}^n(R)$ is the norm of the mean curvature of the generalized cylinder $\partial B^{\mathbb{M}^n(\kappa)}_R\times \mathbb{R}$ in $\mathbb{M}^n(\kappa)\times \mathbb{R}$.
 \end{theorem*}
 \begin{proof}
 Likewise to the proof of theorem \ref{maintheo}, we are using  the test function  $f=F_k\circ r_o\circ\pi\circ\varphi$ with
$$
\begin{aligned}
&F_k:\mathbb{R}\to\mathbb{R},\quad t\to F_k(t)=\int_0^t{\rm sn}_k(s)ds
\end{aligned}
$$
Let us setting,
$$
R^*:=\sup_{\Sigma}r_o(\varphi(\Sigma))<\infty,
$$
then, 
$$
\sup_\Sigma f=\int_0^{R^*}{\rm sn}_k(s)ds<\infty.
$$
By using inequality (\ref{tria2}) of corollary \ref{corunua}
\begin{equation}\begin{aligned}
\triangle_\Sigma f(z)\geq & {\rm sn}_k'(t)\left(1-\Vert d\pi(\nu)\Vert^2\right)+(n-1){\rm sn}_k(t)\frac{{\rm sn}_\kappa'(t)}{{\rm sn}_\kappa(t)}\\ & +n\langle \nabla f,\vec H_\tau\rangle\end{aligned}
\end{equation}  
Since ${\rm sn}'_k(t)\geq 0$, 
\begin{equation}\begin{aligned}
\triangle_\Sigma f(z)\geq & (n-1){\rm sn}_k(t)\frac{{\rm sn}_\kappa'(t)}{{\rm sn}_\kappa(t)} +n\langle \nabla f,\vec H_\tau\rangle\\
=&(n-1){\rm sn}_k(t)\frac{{\rm sn}_\kappa'(t)}{{\rm sn}_\kappa(t)}+\langle\nabla f,\vec H\rangle+\langle\nabla f,\frac{2}{n}\langle \nu,\xi\rangle \nabla\xi(\nu_H)\rangle\\
\geq & (n-1){\rm sn}_k(t)\frac{{\rm sn}_\kappa'(t)}{{\rm sn}_\kappa(t)}+{\rm sn}_k(t)\langle\nabla r_o,\vec H\rangle-\frac{2}{n}\Vert \nabla f\Vert\\
\geq & (n-1){\rm sn}_k(t)\frac{{\rm sn}_\kappa'(t)}{{\rm sn}_\kappa(t)}-{\rm sn}_k(t)\sup_\Sigma\Vert\vec H\Vert-\frac{2}{n}\Vert \nabla f\Vert\\
\end{aligned}
\end{equation} 
Now we can assume that $\sup_\Sigma\Vert \vec{H}\Vert<\infty$, (otherwise, there is nothing to be proved), and hence by using proposition \ref{novaprop5.2}, $\Sigma$ satisfies the Omori-Yau maximum principle for the Laplacian, and $f$ is bounded in $\Sigma$,  there exists a sequence $\{x_i\}$
 $$
 f(x_i)\geq \sup_\Sigma f-\frac{1}{i},\quad \Vert \nabla f(x_i)\Vert<\frac{1}{i},\quad \triangle_\Sigma f(x_i)<\frac{1}{i}.
 $$
 Then, 
 $$
 \frac{1}{i}\geq (n-1){\rm sn}_k(t(x_i))\frac{{\rm sn}_\kappa'(t(x_i))}{{\rm sn}_\kappa(t(x_i))}-{\rm sn}_k(t(x_i))\sup_\Sigma\Vert\vec H\Vert-\frac{2}{ni} 
 $$
letting  $i$ tend to $\infty$, taking into account that $t(x_i)\to R^*$ when $i\to\infty$,
$$
\sup_\Sigma\Vert \vec H\Vert \geq (n-1) \frac{{\rm sn}_\kappa'(R^*)}{{\rm sn}_\kappa((R^*)}\geq h_\kappa(R),
$$
and the theorem is proved.
\end{proof}
\subsection*{Acknowledgments}
I am grateful to professor Pacelli Bessa for his valuable help and for his useful comments and suggestions during the preparation of the present paper.
\def\cprime{$'$} \def\polhk#1{\setbox0=\hbox{#1}{\ooalign{\hidewidth
  \lower1.5ex\hbox{`}\hidewidth\crcr\unhbox0}}}
  \def\polhk#1{\setbox0=\hbox{#1}{\ooalign{\hidewidth
  \lower1.5ex\hbox{`}\hidewidth\crcr\unhbox0}}}
  \def\polhk#1{\setbox0=\hbox{#1}{\ooalign{\hidewidth
  \lower1.5ex\hbox{`}\hidewidth\crcr\unhbox0}}} \def\cprime{$'$}
  \def\cprime{$'$} \def\cprime{$'$} \def\cprime{$'$} \def\cprime{$'$}


\begin{thebibliography}{10}

\bibitem{Alarc2015}
A. Alarc\'on and F. Forstneri{\v c}.
\newblock The {C}alabi-{Y}au problem, null curves, and {B}ryant surfaces.
\newblock {\em Math. Ann.}, 363(3-4):913--951, 2015.

\bibitem{Alarc2008}
A.~Alarcón, L.~Ferrer, and F.~Martín.
\newblock Density theorems for complete minimal surfaces in ℝ3.
\newblock {\em Geometric and Functional Analysis}, 18(1):1--49, 2008.

\bibitem{Alias2009}
L.~J. Al\'ias, G.~P. Bessa, and M. Dajczer.
\newblock The mean curvature of cylindrically bounded submanifolds.
\newblock {\em Math. Ann.}, 345(2):367--376, 2009.

\bibitem{Alias2016}
L.J. Alías, P.~Mastrolia, and M.~Rigoli.
\newblock Maximum principles and geometric applications.
\newblock {\em Springer Monographs in Mathematics}, pages 1--570, 2016.

\bibitem{Atsuji1999}
A.~Atsuji.
\newblock Remarks on harmonic maps into a cone from a stochastically complete
  manifold.
\newblock {\em Proceedings of the Japan Academy Series A: Mathematical
  Sciences}, 75(7):105--108, 1999.

\bibitem{Daniel2007}
B. Daniel.
\newblock Isometric immersions into 3-dimensional homogeneous manifolds.
\newblock {\em Comment. Math. Helv.}, 82(1):87--131, 2007.

\bibitem{Ferrer2012}
L.~Ferrer, F.~Martín, and W.H. Meeks.
\newblock Existence of proper minimal surfaces of arbitrary topological type.
\newblock {\em Advances in Mathematics}, 231(1):378--413, 2012.

\bibitem{GriExp}
A. Grigor{\cprime}yan.
\newblock Analytic and geometric background of recurrence and non-explosion of
  the {B}rownian motion on {R}iemannian manifolds.
\newblock {\em Bull. Amer. Math. Soc. (N.S.)}, 36(2):135--249, 1999.

\bibitem{JorgeXavier80}
L. P. de~M. Jorge and F. Xavier.
\newblock A complete minimal surface in {${\bf R}^{3}$}\ between two parallel
  planes.
\newblock {\em Ann. of Math. (2)}, 112(1):203--206, 1980.

\bibitem{Lerma2017}
A.~M. Lerma and J. M. Manzano.
\newblock Compact stable surfaces with constant mean curvature in {K}illing
  submersions.
\newblock {\em Ann. Mat. Pura Appl. (4)}, 196(4):1345--1364, 2017.

\bibitem{Manzano2017}
J.M. Manzano and B.~Nelli.
\newblock Height and area estimates for constant mean curvature graphs in
  $\mathbb{E}(\kappa,\tau)$-spaces.
\newblock {\em Journal of Geometric Analysis}, 27(4):3441--3473, 2017.

\bibitem{Morales1}
F. Mart{\'{\i}}n and S. Morales.
\newblock Complete proper minimal surfaces in convex bodies of {$\mathbb{R}^3$}.
\newblock {\em Duke Math. J.}, 128(3):559--593, 2005.

\bibitem{Morales2}
F. Mart{\'{\i}}n and S. Morales.
\newblock Complete proper minimal surfaces in convex bodies of {$\mathbb{R}^3$}.
  {II}. {T}he behavior of the limit set.
\newblock {\em Comment. Math. Helv.}, 81(3):699--725, 2006.

\bibitem{Nadi}
N. Nadirashvili.
\newblock Hadamard's and {C}alabi-{Y}au's conjectures on negatively curved and
  minimal surfaces.
\newblock {\em Invent. Math.}, 126(3):457--465, 1996.

\bibitem{OneillSubmersion}
B. O'Neill.
\newblock The fundamental equations of a submersion.
\newblock {\em Michigan Math. J.}, 13:459--469, 1966.

\bibitem{Oneill}
B. O'Neill.
\newblock {\em Semi-{R}iemannian geometry}, volume 103 of {\em Pure and Applied
  Mathematics}.
\newblock Academic Press Inc. [Harcourt Brace Jovanovich Publishers], New York,
  1983.
\newblock With applications to relativity.

\bibitem{Petersen}
P. Petersen.
\newblock {\em Riemannian geometry}, volume 171 of {\em Graduate Texts in
  Mathematics}.
\newblock Springer-Verlag, New York, 1998.

\bibitem{Pigola2003}
S. Pigola, M. Rigoli, and A. ~G. Setti.
\newblock A remark on the maximum principle and stochastic completeness.
\newblock {\em Proc. Amer. Math. Soc.}, 131(4):1283--1288, 2003.

\bibitem{Scott1983}
P. Scott.
\newblock The geometries of {$3$}-manifolds.
\newblock {\em Bull. London Math. Soc.}, 15(5):401--487, 1983.

\bibitem{Tokuomaru2007}
M.~Tokuomaru.
\newblock Complete minimal cylinders properly immersed in the unit ball.
\newblock {\em Kyushu Journal of Mathematics}, 61(2):373--394, 2007.

\end{thebibliography}
\end{document}